\documentclass[
12pt,reqno]{amsart}
\usepackage{verbatim}
\usepackage{amssymb}
\usepackage{amsrefs}

\usepackage{enumerate}
\usepackage[active]{srcltx}
\numberwithin{equation}{section}

\usepackage{t1enc}
\usepackage[utf8x]{inputenc}

\newtheorem{theorem}{Theorem}[section]
\newtheorem{lemma}[theorem]{Lemma}
\newtheorem{corollary}[theorem]{Corollary}

\newtheorem{problem}[theorem]{Problem}
\newtheorem{proposition}[theorem]{Proposition}

\newtheorem{claim}{Claim}[theorem]

\theoremstyle{definition}
\newtheorem{definition}[theorem]{Definition}

\theoremstyle{remark}
\newtheorem{remark}{Remark}  

\newcommand{\itemprefix}{}
\makeatletter
\newcommand{\myitem}{%
\item\protected@edef\@currentlabel{\itemprefix\theenumi}%
}
\makeatother

 \newcommand{\dom}{\operatorname{dom}}

\newcommand{\supp}{\operatorname{supp}}
\newcommand{\val}{\operatorname{val}}

\author[A. Dow]{Alan Dow}
\address{Department of Mathematics and Statistics,
University of North Carolina at Charlotte,
 Charlotte, NC 28223, USA}
\email{adow@charlotte.edu}

\author[I. Juh\'asz]{Istv\'an Juh\'asz}
\address      {HUN-REN Alfr\'ed Rényi Institute of Mathematics}
\email{juhasz@renyi.hu}



\title{All spaces of countable spread can be small}
\date{\today}

\begin{document}

\begin{abstract}
The main result of this paper is the proof of the simultaneous consistency, modulo a weakly compact cardinal,
of the equality $2^{< \mathfrak{c}} = \mathfrak{c}$ with the following property (*) of partitions of pairs of $\mathfrak{c}$:

\smallskip

(*) For any
coloring (or partition) $k : [\mathfrak{c}]^2 \rightarrow 2$ either there is a homogeneous set
of size $\mathfrak{c}$ in color $0$ or there is a set $S \in [\mathfrak{c}]^\mathfrak{c}$
such that for every countable $A \subset S$ there is $\beta \in \mathfrak{c}$ for which $A \subset \beta$
and $k(\{\alpha, \beta\}) = 1$ for all $\alpha \in A$.

\smallskip

(*) plus $2^{< \mathfrak{c}} = \mathfrak{c}$ together then imply that for every topological space $X$ of countable
  spread, i.e. not containing any uncountable discrete subset,
   $|X| \le \mathfrak{c}$ if it is
  Hausdorff and
  $o(X) = \mathfrak c$ if it is also infinite and regular. Here $o(X)$ denotes the number of all open subsets of $X$.
\end{abstract}

\maketitle
\section{Introduction}
We recall that the spread $s(X)$ of a topological space is the supremum of the cardinalities
of all discrete subspaces of $X$. Hajnal and Juhász proved in \cite{MR0229195} that $|X| \le 2^{2^{s(X)}}$
for any Hausdorff space $X$. In particular, $|X| \le 2^\mathfrak{c}$ whenever $X$ has countable spread.
It is also known that any {\em regular} space $X$ of countable spread has a network of size $\le \mathfrak{c}$, and
that clearly implies $$w(X) \le o(X) \le 2^\mathfrak{c}.$$

In \cite{MR0852486} Juhász and Shelah proved that it is consistent with both $\mathfrak{c}$ and $2^\mathfrak{c}$
being arbitrarily large, independently of each other, that there are 0-dimensional Hausdorff, hence regular
hereditary separable (resp. hereditary Lindelöf) spaces of cardinality (resp. weight) $2^\mathfrak{c}$ (see also the last section of this paper).
This shows that, at least up to consistency, the above inequalities are sharp.

In the opposite direction, Todor\v cevi\' c proved in \cite{TodorcevicTAMS} that it is consistent to have
$|X| \le \mathfrak{c}$ for any Hausdorff  space $X$ of countable spread. But as far as we know, the analogous
result for the second inequality has not been known. In particular, it was unknown if it is consistent
to have $w(X) \le \mathfrak{c}$ for any regular hereditary Lindelöf space $X$.
(Note that if $X$ is hereditary Lindelöf then $w(X) \le \mathfrak{c}$ clearly implies $o(X)\le \mathfrak{c}$.)
Actually, this
question was the original motivation for our present work.

Our main result says that, at least modulo the existence of a weakly compact cardinal,
we do get the consistency of $o(X) = \mathfrak{c}$ for any infinite regular space $X$
of countable spread together with $|X| \le \mathfrak{c}$
for any Hausdorff  space $X$ of countable spread. This yields a counterpoint of the results in \cite{MR0852486} and
justifies the title of our paper.

Our notation and terminology follow those in \cite{MR0576927}. In what follows, {\em space} will always
mean an infinite Hausdorff space.

\section{A partition relation implying that spaces of countable spread are small}

We start by introducing a partition relation that will be used to get what we want and
whose consistency will be established later.
\begin{definition}\label{firstversion}
(i) If $A$ and $B$ are disjoint sets of ordinals then $$(A;B) = \big \{\{\alpha, \beta\} : \alpha \in A,\,\beta \in B \text{ and } \alpha < \beta \big \}.$$

(ii) An {\em $(\omega_1;\omega_1)$-type half graph} is a set of
pairs of the form $(A;B)$ where $A$ and $B$ are disjoint sets of ordinals
of order type $\omega_1$ such that $\sup A = \sup B$.

(iii) $\kappa \rightarrow (\omega_1, (\omega_1;\omega_1))$ denotes the statement that for any
coloring (or partition) $k : [\kappa]^2 \rightarrow 2$ either there is a homogeneous subset of $\kappa$
of size $\omega_1$ in color $0$ or there is an $(\omega_1;\omega_1)$-type half graph homogeneous in color $1$.
\end{definition}

It is well-known that a space $X$ is hereditary separable (resp. hereditary Lindelöf) iff
every left separated (resp. right separated) subspace of $X$ is countable. We shall need
a variation of these concepts that, however, for regular spaces coincide with the old ones.

\begin{definition}
The space $X$ is {\em strongly left separated} (resp. {\em strongly right separated}) if it has
a well-ordering $\prec$ such that every point $x \in X$ has a {\em closed neighborhood} $W_x$ such
that $y \notin W_x$ whenever $y \prec x$ (resp. $x \prec y$).

We say that $X$ is {\em strongly separated} if it is strongly left separated or strongly right separated.
\end{definition}

Clearly, every strongly separated space is automatically Hausdorff.

\begin{proposition}\label{pr:arrow}
The principle
$\kappa \rightarrow (\omega_1, (\omega_1;\omega_1))$
  implies that  every strongly separated space of cardinality
$\kappa$ has an uncountable
 discrete subspace.
\end{proposition}

\begin{proof}
(i) Assume we have a strongly left separated
topology on $\kappa$. We may assume that the
strongly left separating well-ordering
is the standard ordering of $\kappa$.
For each $\xi\in \kappa$,
choose a closed neighborhood $W_\xi$ of $\xi$
that is contained in $\kappa  \setminus \xi$.
Define the coloring $k : [\kappa]^2 \rightarrow 2$ according to $
k(\{\xi,\eta\}) = 1$ providing $\eta\in W_\xi \setminus \{\xi\} $.

Naturally, if $I\subset \kappa$ is homogeneous of color $0$
  then $I$ is discrete.

  So assume that $A = \{ \alpha_\xi : \xi<\omega_1\}$
  and $B = \{\beta_\xi : \xi < \omega_1\}$,
  indexed in their increasing ordering,
  are subsets of
  $\kappa$ such that $(A;B) \subset k^{-1}(1)$ forms an $(\omega_1;\omega_1)$-type half graph.
  By thinning out $B$ if necessary, we may assume that
  for each $\xi\leq \eta\in \omega_1$ we have
$\{\alpha_\xi, \beta_\eta\} \in k^{-1}(1)$.
We check that then $\{\beta_{\xi} : \xi\in \omega_1\}$ is
a free sequence in $X$.

Fix any $\delta\in\omega_1$,
and let $\gamma_\delta=\sup\{\beta_\xi : \xi < \delta\}$.
Since $X$ is left separated, the closure of
 $\{\beta_\xi : \xi<\delta\} $ is a subset of $\gamma_\delta$.
 The closure of $\{\beta_\xi : \delta\leq \xi\}$ is a subset
 of the closed set $W_{\alpha_\delta}$ which is contained
 in $[\alpha_\delta,\kappa)$. Since $\gamma_\delta\leq \alpha_\delta$,
  the sequence $\{ \beta_\xi : \xi < \omega_1\}$ is a free sequence.

(ii) Now assume that we have a strongly right separated
topology on $\kappa$.
Fix for each $\eta \in \kappa$ a closed neighborhood
$W_\eta \subset [0,\eta]$.
Again let $k : [\kappa]^2 \rightarrow 2$ be defined by
 $k(\{\xi,\eta\}) = 1$ iff $\xi<\eta$ and
$\xi\in W_\eta$.

If $I \subset \kappa$ is uncountable with $[I]^2 \subset k^{-1}(0)$
then again $I$ is discrete. Otherwise we may
 assume that there are $A = \{ \alpha_\xi : \xi<\omega_1\}$
  and $B = \{\beta_\xi : \xi < \omega_1\}$ subsets of
  $\kappa$ satisfying that, for each $\xi\leq \eta\in \omega_1$ we have
$\{\alpha_\xi, \beta_\eta\} \in k^{-1}(1)$.

We check that now $\{\alpha_{\xi} : \xi\in \omega_1\}$ is
a free sequence in $X$. Fix any $\delta\in\omega_1$,
and let $\gamma_\delta=\sup\{\alpha_\xi : \xi < \delta\}$.
Since $X$ is right separated, the closure of
 $\{\alpha_\xi : \delta \leq \xi \} $ is a subset of
 $[\alpha_\delta,\kappa)$.
 The closure of $\{\alpha_\xi : \xi < \delta \}$ is a subset
of $W_{\beta_\delta}$ which is contained
 in $[0,\beta_\delta]$.  Since $\beta_\delta < \alpha_\delta$,
 this proves that indeed
  the sequence $\{\alpha_\xi : \xi < \omega_1\}$ is a free sequence.
\end{proof}

 \begin{proposition}\label{pr:oc}
 Assume  that there is an open cover $\mathcal{U}$ of the space $X$ such that for every
 $\mathcal{V} \subset \mathcal{U}$ with $|\mathcal{V}| < \kappa$ we have $\bigcup \{\overline{V} : V \in \mathcal{V}\} \ne X$.
 Then $X$ contains a strongly right separated subspace of cardinality $\kappa$.
\end{proposition}

\begin{proof}
We may choose by a straight-forward transfinite recursion points $x_\alpha \in X$ and members $U_\alpha$
of $\mathcal{U}$ for $\alpha < \kappa$ such that $x_\alpha \in U_\alpha$ and $x_\beta \notin \overline{U_\alpha}$
whenever $\alpha < \beta$. Clearly, then $\{x_\alpha : \alpha < \kappa\}$ is strongly right separated.
 \end{proof}

\begin{theorem}\label{tm:T2T3}
Assume we have $2^{<\mathfrak{c}} = \mathfrak{c}$,
and, moreover,
that every strongly separated space of cardinality $\mathfrak{c}$
has an uncountable discrete subspace. Then
\begin{enumerate}
  \item every Hausdorff space of countable spread has cardinality $\le \mathfrak{c}$;
  \item for every infinite regular space $X$ of countable spread we have $o(X) = \mathfrak{c}$.
\end{enumerate}
\end{theorem}

\begin{proof}

(1) Let $X$ be any Hausdorff space of countable spread. We first show that every point $x \in X$
  has pseudo-character $\psi(x,X) < \mathfrak{c}$. To see this, we choose for every point $y \in X \setminus \{x\}$
  an open neighborhood $U_y$ such that
  $x \notin \overline{U_y}$. this is possible because $X$ is Hausdorff. We may now apply Proposition \ref{pr:oc}
  to the open cover $\{U_y : y \in X \setminus \{x\}\}$ of the space $X \setminus \{x\}$ to find
  a subset $Y \subset X \setminus \{x\}$ such that $|Y| < \mathfrak{c}$ and $X \setminus \{x\} = \bigcup \{\overline{U_y} : y \in Y\}$.
  This clearly implies $\psi(x,X) \le |Y| < \mathfrak{c}$.

  Now, for every infinite cardinal $\lambda < \mathfrak{c}$ let us put $$X_\lambda = \{x \in X : \psi(x,X) \le \lambda\}.$$
  Since $|Z| \le 2^{s(Z)\cdot\psi(Z)}$ holds for any $T_1$-space $Z$, see e.g. \cite{MR0229195} or \cite{MR0576927},
  we get using $2^\lambda = \mathfrak{c}$ that
  $|X_\lambda| \le \mathfrak{c}$. But we have $X = \bigcup \{X_\lambda : \omega \le \lambda < \mathfrak{c}\}$ by the above,
  consequently $|X| \le \mathfrak{c}$ as well.

(2)  Now, let $X$ be any infinite regular space $X$ of countable spread. We know by (1) that $|X| \le \mathfrak{c}$.
  It is also well-known that every topological space $Y$ has a left separated subspace of size $d(Y)$, the density
  of $Y$. As $X$ is regular, its left separated subspaces are actually strongly left separated, hence
  it follows from our assumptions that the density of any subspace of $X$ is $< \mathfrak{c}$. In particular, then
  every closed subset $F$ of $X$ is of the form $F = \overline{A}$, where $A \in [X]^{< \mathfrak{c}}$. But
  $|X| \le \mathfrak{c}$ and $2^{<\mathfrak{c}} = \mathfrak{c}$ imply $|[X]^{< \mathfrak{c}}| = \mathfrak{c}$,
  hence the number of all closed subsets of $X$, that is $o(X)$, is at most $\mathfrak{c}$. But for every infinite
  Hausdorff space $X$ we obviously have $o(X) \ge \mathfrak{c}$, consequently $o(X) = \mathfrak{c}$.
\end{proof}

\begin{corollary}  Under the assumptions of Theorem \ref{tm:T2T3} every  regular space of countable spread
has weight and cardinality at most $\mathfrak c$. In particular, every regular hereditary Lindelöf space
has weight $\le \mathfrak c$.
\end{corollary}

By Proposition \ref{pr:arrow} the assumptions of Theorem \ref{tm:T2T3} follow if we have both $2^{<\mathfrak{c}} = \mathfrak{c}$
and $\mathfrak{c} \rightarrow (\omega_1, (\omega_1;\omega_1))$. The next two sections are devoted to proving that even
the consistency of $2^{<\mathfrak{c}} = \mathfrak{c}$
and $\mathfrak{c} \rightarrow (\mathfrak{c}, (\omega_1;\omega_1))$
can be established, modulo the existence of a weakly compact cardinal.

Next we point out that Theorem \ref{tm:T2T3} implies that every regular $\Delta$-space
has cardinality $< \mathfrak{c}$. This fact, that is immediate from Corollary 4.2 of \cite{JvMSSz},
yields a consistent affirmative answer to Problem 4.4 of \cite{JvMSSz} in the regular case.

We end this section by formulating several problems that we think deserve further research.

\begin{problem}
Does the consistency of part (2) of Theorem \ref{tm:T2T3} require large cardinals?
Does the consistency of all hereditary Lindelöf regular spaces having weight $\le \mathfrak{c}$ require large cardinals?
\end{problem}

\begin{problem}
Is part (1) or (2) of Theorem \ref{tm:T2T3} consistent with CH?
\end{problem}

We note here that $\omega_1 \rightarrow (\omega_1, (\omega_1;\omega_1))$ would imply that there are no $L$-spaces,
hence it is provably false in ZFC, see \cite{MR2220104}.

\begin{problem}
Is it consistent to extend part (2) of Theorem \ref{tm:T2T3} from regular to Hausdorff spaces?
\end{problem}

\begin{problem}
Is it consistent to have $|X| \le 2^{s(X)}$ for every regular (or even Hausdorff) space $X$?
Is it consistent to have $w(X) \le 2^{hL(X)}$ for every regular space $X$?
\end{problem}

\section{on the forcing}
Let $\kappa$ be a cardinal.  We will produce
a model using the following poset that was
introduced in \cite{FuchinoShelahSoukup} and was called there the {\em pseudo product}
of $\kappa$ copies of the Cohen forcing.

\begin{definition}
The poset $P_\kappa=\Pi^*_{\alpha\in\kappa} 2^{<\omega}$ consists
of
the set of countable partial functions
from $\kappa$ to $2^{<\omega}$.

For $p,q\in P_\kappa$, we define
 $p< q$ providing
\begin{enumerate}
\item $\dom(p)\supset \dom(q)$,
\item for all $\alpha\in \dom(q)$,
   $q(\alpha)\subset p(\alpha)$,
   \item  $\{\beta\in \dom(q):
    p(\beta)\neq q(\beta)\}$ is finite.
\end{enumerate}
\end{definition}

The following property of $P_\kappa$ is taken from
\cite{FuchinoShelahSoukup}*{Corollary 2.4(a), Theorem 2.6}
\begin{lemma}
 The poset $P_\kappa$ is\label{poset} proper, has cardinality
  $\kappa^{\aleph_0} $ and satisfies
 the $\mathfrak c^+$-cc.
\end{lemma}

 By a standard counting of nice names argument,
 this next result follows from Lemma \ref{poset}.

\begin{proposition} If $ \kappa$ is a strongly inaccessible
cardinal, then $2^{<\mathfrak c} = \mathfrak c=\kappa$ holds
in the forcing extension by $P_\kappa$.
\end{proposition}

The following property of this poset is obvious but is crucial
to our forcing result. This property is stronger
than simply having that the conditions have countable support.

\begin{proposition}  If $\{ p_n : n\in \omega\}\subset
 P_\kappa$ have disjoint domains, then $\bigcup_n p_n$ is
 a common extension of each $p_n$ in $P_\kappa$.
\end{proposition}

\section{on the consistency of $\mathfrak c \rightarrow
 (\omega_1, (\omega_1;\omega_1))$}

Throughout this section, let $\kappa$ be a weakly compact cardinal.
We prove this next theorem:

\begin{theorem} If $\kappa$ is\label{main*} a weakly compact cardinal, and if
 $G$ is a $P_\kappa$-generic filter, then
 $\kappa \rightarrow (\kappa, (\omega_1;\omega_1))$ plus
  $2^{<\kappa}=\kappa=\mathfrak c$ hold
  in $V[G]$.
\end{theorem}

Of course it immediately begs the question as to whether one
needs that $\kappa$ is weakly compact (or a large cardinal)
to obtain a similar result.
To this end we note that we obtain a stronger result.

\begin{definition} We define a statement (*):\\
(*) For any
coloring $k : [\mathfrak{c}]^2 \rightarrow 2$ either there is a homogeneous set
of size $\mathfrak{c}$ in color $0$ or there is a set $S \in [\mathfrak{c}]^\mathfrak{c}$
such that for every countable $A \subset S$ there is $\beta \in \mathfrak{c}$ for which $A \subset \beta$
and $k(\{\alpha, \beta\}) = 1$ for all $\alpha \in A$.
\end{definition}

\begin{proposition}  The statement (*) implies
that $\mathfrak c \rightarrow (\mathfrak{c},(\omega_1;\omega_1))$.
\end{proposition}

\begin{proof} Let $k$ be a 2-coloring of $[\mathfrak c]^2$
and assume there is no 0-homogeneous set of size $\mathfrak{c}$.
 Then choose a subset $S$ of $\mathfrak{c}$ of cardinality $\mathfrak c$
 as in the statement of (*). Using that $cf(\mathfrak{c}) > \omega$,
we may then easily choose by recursion sets $A = \{\alpha_\xi : \xi < \omega_1\} \subset S$
and $B = \{\beta_\xi : \xi < \omega_1\} \subset \mathfrak{c}$ with $\sup A = \sup B$
such that $(A;B) \subset k^{-1}(1).$
\end{proof}

 \begin{remark} The statement ``$2^{<\mathfrak c}=\mathfrak c  ~\& ~
  (*)$''
 does  have some large cardinal strength.
 To see this, we note that if there is a
 $\mathfrak c$-Souslin tree ordering $\prec$ on $\mathfrak c$, then
 the statement (*)
 fails. To see this let $k$ be the function
 from $[\mathfrak c]^2$  into $\{0,1\}$
 satisfying that $k(s,t)=1$ if and only if $s\prec t$.
Obviously  a $0$-homogeneous set is an anti-chain and
if the set $S $ has cardinality $\mathfrak c$
and satisfies the  1-homogeneous requirement of (*),
 then $S$ is a chain.
It was proven in  \cite{BrRinot2017},
that if
 $\mathfrak c =\kappa = \lambda^+=2^\lambda$
and $\square_\lambda$ holds
then  there is a $\kappa$-Souslin tree $T$.
\end{remark}

We will need this next property of a weakly compact cardinal,
see
 \cite{KunenHandbookLogic}*{Lemma 6.8}.

 \begin{proposition} If $f$ is any function from
   $[\kappa]^{3}$ into a set $S$ of cardinality less $\kappa$,
    then there is a $\kappa$-sized set
  $U\subset \kappa  $ that is homogeneous
    for $f$.
    \end{proposition}

 The remainder of this section is devoted to proving
 this next theorem.

\begin{theorem} If $\kappa$ is a weakly compact cardinal, and if
 $G$ is a $P_\kappa$-generic filter, then
 (*)      holds
  in $V[G]$.
\end{theorem}

\begin{proof}
Let $\dot E$ be a $P_\kappa$-name of a subset
of $[\kappa]^2 $.  Since $P_\kappa$ is homogeneous,
 it suffices to prove that $1$ does not force that $\dot E$,
 i.e. the characteristic function of $\dot E$ in $[\kappa]^2 $,
 is a counterexample to
 (*).
 Let $H$ denote the set of
 all $\{\xi,\eta\}_<\in [\kappa]^2$  such that there exists a condition
  $p(\xi,\eta)$ that forces that $\{  \xi,\eta\}\in \dot E$.
  We assume with no loss of generality that $\{\xi,\eta\}
 \subset \dom(p(\xi,\eta))$.
  For any $\xi<\eta$ with $\{  \xi,\eta\}\notin H$,
  let $p(\xi,\eta)$ be the set $\{\xi,\eta\}$ (i.e. simply
  anything not in $P_\kappa$).
    \bigskip

 For $\{\xi,\eta\}_<\in H$, let $I(\xi,\eta)$ be the domain
 of $p(\xi,\eta)$. Let $\delta(\xi,\eta)$ denote the order-type
 of $I(\xi,\eta)$ and let $h_{\xi,\eta} $ be the order-preserving
 function from $\delta(\xi,\eta)$ to $I(\xi,\eta)$.
 For any $\{\xi,\eta\}_{<}\in [\kappa]^2\setminus H$,
 let $I(\xi,\eta)$ and $h_{\xi,\eta}$ be the empty set.

 \medskip

 We define a function $f$ on $[\kappa]^3$
 where $f(\{\xi,\eta,\zeta\}_{<})$ is sent to the isomorphism type
 of
 $$\langle p(\xi,\eta), p(\xi,\zeta), p(\eta,\zeta),
  \delta_{\xi,\eta}, \delta_{\xi,\zeta},\delta_{\eta,\zeta},
   h_{\xi,\eta}, h_{\xi,\zeta}, h_{\eta,\zeta}\rangle~.$$
 Fix any $U\subset \kappa$ of size $\kappa$ that is homogeneous
 for $f$.
 \bigskip

 Case 1: there are $\xi < \eta<\zeta<\kappa$ in $U$ such that
  $\{\eta,\zeta\}\notin H$.
  \medskip

  In this case, it is clear that $1$ forces that if $\bar\xi=\min( U)$,
   then $[U\setminus\{\bar\xi\}]^2 \cap H$ is empty, and therefore
   that $1$ forces that $[U\setminus \{\bar\xi\}]^2 \cap \dot E$ is empty.
   Thus, in Case 1, $U$ is a 0-homogeneous set of cardinality $\mathfrak c$.
   \bigskip

 Case 2: $[U]^2 \subset H$.

 \medskip

 Clearly there is a $\delta<\omega_1$ such that for all
  $\xi < \eta\in U$, $\delta_{\xi,\eta}=\delta$.
  Let $\bar\xi = \min(U)$.
  \medskip

We are interested in partitioning
each $I(\xi,\zeta )$ (for $\xi<\eta<\zeta\in U$)
 into the three pieces consisting of $I(\xi,\zeta)\cap \xi$,
 and loosely speaking
  the part of $I(\xi,\zeta)\setminus \xi$ that is close to $\xi$
  and the part that is close to $\zeta$.

   \medskip

For this purpose, fix any $\bar\xi< \eta < \zeta $ all in $U$, and
  let $\beta_1< \delta$ denote the element such that
  $\bar\xi=h_{\bar\xi,\zeta}(\beta_1)$. Similarly
  let $\beta_2<\delta$ be minimal such
  that $\eta \leq h_{\bar\xi,\zeta}(\beta_2)$.
  Finally, let $L=\{\beta\in \delta\setminus \beta_2 :
   h_{\bar\xi,\zeta}(\beta) = h_{\eta,\zeta}(\beta)\}$.
  \medskip

   \begin{claim} For\label{wc1} all $\xi < \eta < \zeta$ in $U$, the
   following statements hold:
   \begin{enumerate}
  \item     $I(\xi,\eta)\subset          h_{\xi,\zeta}(\beta_2)\leq \zeta$,
  \item for all $\beta <\beta_1$, $h_{\xi,\eta}(\beta) = h_{\eta,\zeta}(\beta)$,
  \item for all $\beta_1\leq \beta < \beta_2$, $h_{\xi,\eta}(\beta) =
   h_{\xi,\zeta}(\beta)$,
   \item for all $\beta\in L$, $h_{\xi,\zeta}(\beta) = h_{\eta,\zeta}(\beta)$,
   \item for all $\beta\in  \delta\setminus (\beta_2\cup L)$,
      $h_{\xi,\zeta}(\beta) \notin I(\eta,\zeta)$.
      \end{enumerate}
   \end{claim}

 \begin{proof}
 Since $\kappa$ is weakly compact, there is a
 regular cardinal $\mu\in \kappa$
 such that the order-type of $U\cap \mu$ equals $\mu$
 and such that $I(\xi,\eta)\subset \mu$ for all $\xi<\eta$
 in $U\cap \mu$.  Let $\gamma = \min(U\setminus \mu)$.
\medskip

 Consider any $\xi < \eta$ both in $U\cap \mu$.
   Choose $\zeta \in U\cap \mu$ so that $I(\xi,\eta) \subset \zeta$.
   Since $U$ is homogeneous for $f$, it suffices to prove each of
   the items for the triple $(\xi,\eta,\zeta)$.
Here are the proofs of each item:
\begin{enumerate}
\item     Apply the definition of $\beta_2$ to the triple
    $(\xi,\zeta,\gamma)$, we have that $\zeta\leq h_{\xi,\gamma}(\beta_2)$.
    Jumping to $(\xi,\eta,\gamma)$, we have
    that $I(\xi,\eta) \subset h_{\xi,\gamma}(\beta_2)$.
    Therefore, by homogeneity, we have that $I(\xi,\eta)\subset h_{\xi,\zeta}(\beta_2)$.
\item Fix any $\beta < \beta_1$.  It clearly   follows
from the homogeneity that the sequence $\{ h_{\xi,\eta'}(\beta):  \xi<\eta'\in U\}$
 is monotone increasing. Since $h_{\xi,\eta'}(\beta) <\xi$ for all $\eta'\in U$,
  it follows that the sequence is constant.
  \item Again, for any $\beta_1\leq \beta < \beta_2$,
    the sequence $\{ h_{\xi,\eta'}(\beta) : \xi<\eta'\in U\cap \mu\}$
    is monotone increasing.  On the other hand,  $\beta_2$ is minimal
    such that
    $h_{\xi,\gamma}(\beta_2)$ is greater than $\eta$ and the sequence
    $\{ h_{\xi,\eta'}(\beta) : \xi<\eta'\in U\cap \mu\}$ is bounded
    by $h_{\xi,\gamma}(\beta) < \eta$ and
    so it can not be strictly increasing
    with cardinality $\mu$.
\item The fact that,
for all $\beta\in L$, $h_{\xi,\zeta}(\beta) = h_{\eta,\zeta}(\beta)$
is a direct consequence of homogeneity.
\item Assume that $\beta_2\leq \beta\in \delta$, and
that $h_{\xi,\zeta}(\beta)\in I(\eta,\zeta)$. Choose $\beta_2
\leq\alpha < \delta$ such that $h_{\eta,\zeta}(\alpha) =
 h_{\xi,\zeta}(\beta)$. By homogeneity, applied to
  $\xi<\eta <\gamma$,  it follows that $h_{\xi,\gamma}(\beta)
   = h_{\eta,\gamma}(\alpha)$. Similarly, it follows
   that $h_{\xi,\gamma}(\beta) = h_{\zeta,\gamma}(\alpha)$.
   Therefore, $h_{\eta,\gamma}(\alpha) = h_{\zeta,\gamma}(\alpha)$.
   By homogeneity, it follows that $\alpha\in L$
   and so $h_{\xi,\zeta}(\alpha) = h_{\eta,\zeta}(\alpha)$.
   Since $h_{\xi,\zeta}$ is one-to-one, it finally
   follows that $\beta=\alpha$ and is in $L$.
    \end{enumerate}
This proves the Claim.
   \end{proof}

It follows from item 1 of the Claim, that
there is a condition $p\in P_{\bar\xi}$ that
is equal to $p(\xi,\eta)\restriction \xi$
for
all $\xi<\eta$ in $U$.  It follows from item 3 of the Claim
that, for all $\xi\in U$, there is a condition
   $p_\xi\in P_{\kappa\setminus \xi}$ that is
   equal to $p(\xi,\eta) \restriction \eta\setminus \xi$
   for all $\xi<\eta\in U$. Finally, for all $\bar\xi<\zeta\in U$,
   let $J_\zeta  = h_{\bar\xi,\zeta}[L]$ and
   $q_\zeta = p(\bar\xi,\zeta)\restriction J_\zeta$.
   It follows
   from the definition of $L$ that $q_\zeta$ is
   equal to $p(\xi,\zeta)\restriction J_\zeta$
   for all $\xi\in U\cap \zeta$ and $\zeta\in U$.

\medskip

This next claim follows immediately from item 5 of the previous
Claim.

 \begin{claim} For every $\bar\xi<\zeta\in U$,
  the family $\{  I(\eta,\zeta) \setminus (J_\zeta\cup
    h_{\eta,\zeta}(\beta_2)) : \eta\in U\cap \zeta\}$
    is pairwise disjoint.
    \end{claim}

  Now we can complete the proof of the theorem.  Let
  $G$ be any generic filter with $p\in G$. Let $\dot S$
  be the name $\{ (\eta,p_\eta) : \eta\in U\}$. It
  follows from item 1 of Claim \ref{wc1}  that the domains
  of the family $\{ p_\eta : \bar\xi < \eta \in U\}$ are pairwise
  disjoint, and therefore $G\cap \{ p_\eta : \bar\xi< \eta\in U\}$
  has cardinality $\kappa$.  This proves that $S = \val_{G}(\dot S)$
has cardinality $\kappa$. Also let $E = \val_G(\dot E)$.
Similarly, let $\dot B $ be the name
 $\{ \zeta, q_\zeta) : \zeta\in U\}$ and $B  = \val_G(\dot B)$.
   We also have, by the same
  reasoning as with $S$, that $B$ has cardinality $\kappa$.
\medskip

Since $P_\kappa$ is proper, every countable subset of $S$ is
contained in a countable subset of $U$ from the ground model.
Let $A$ be any countable subset of $U$. We prove
that $(A \cap S;  \{\beta\})\subset E$
for some
  $\beta\in B\setminus\sup(A')$.
 \medskip

 For each $\sup(A ) < \zeta\in U$, let $r(A ,\zeta)$
 equal the union of the conditions
  $\{ p(\eta,\zeta)\restriction (I(\eta,\zeta)\setminus
    (J_\zeta\cup h_{\eta,\zeta}(\beta_2))) : \eta\in A'\}$.
    Since these domains are pairwise  disjoint,
     each $q_\zeta \cup r(A ,\zeta)$ is an element of $P_\kappa$.

     There is certainly a $\kappa$-sized subfamily of
      $\{ q_\zeta\cup r(A ,\zeta) : \sup(A )<\zeta\in U\}$
      such that the domains form a pairwise disjoint family.
      It follows that there is some $\sup(A )<\beta\in U$ such
      that $q_\beta\cup r(A ,\beta)$ is in $G$. It also
      follows that $\beta\in B$.
      \medskip

      Finally let $\tilde A = A \cap S$ in $V[G]$.
      It follows that, for each $\eta\in \tilde A$,
       each of $p, p_\eta$ and $q_\beta\cup r(A ,\beta)$
       are in $G$. Fix any $\eta\in \tilde A$,
        the condition $p\cup p_\eta\cup q_\beta\cup r(A ,\beta)$
        is an extension of $p(\eta,\beta)$. This implies
        that $\{\eta,\beta\}\in     E $ for
        every $\eta\in \tilde A$ as required.
\end{proof}

\section{what if $\mathfrak c$ is rvm? }

In this section we prove that in Solovay's original model,
as in \cite{Solovay}*{Theorem 7},
in which   $\mathfrak c$ is real-valued measurable,
there is a hereditarily Lindel\"of space of weight
greater than $\mathfrak c$.   It is easily seen
that the result also holds if Cohen reals are added in place
of random reals. This result seems interesting for a couple
of reasons. The first is to show that forcing
with a  proper poset that
allows countable supports (as is true for random reals) is
not sufficient to obtain $\mathfrak c \rightarrow
 (\omega_1, (\omega_1;\omega_1))$. The second is that
  $\mathfrak c$ has the tree-property in these models
  (see \cite{Kanamori}*{Theorem 7.12}), and
  so this shows that this is not sufficient either.
\medskip

Let $\kappa$ be any uncountable cardinal and let
$T$ denote the set of all $t\in 2^{<\kappa}$
such that $\omega\subset \dom(t)$.
 Let $\mathcal M_T$ denote the usual measure algebra on
 $2^T$. Let $T^+$ denote all the $\kappa$-branches
 of $T$; we use $T^+$ rather than $2^\kappa$ to make
 it clearer that it refers to ground model branches.
 If $\kappa$ is a measurable cardinal, then
  $\mathfrak c$ will be a real-valued measurable
  cardinal in the forcing extension, but this is
  not needed for the properties of the example.
 \medskip

 For each $\kappa$-branch $\rho\in  T^+$,
 we define the simple name, $\dot W_\rho$,
 for a subset of $\kappa\setminus \omega$.
 For each $\alpha\in\kappa\setminus \omega$,
 the clopen set in $2^T$ given by
 $[(\rho\restriction\alpha,1)]$ (usual notation for
 basic clopen subsets of a product) is equal
 to the Boolean forcing value
  $[[\alpha\in \dot W_\rho]]$ (which has
  measure $\frac 12$; we will use $\nu$ for
  the measure on $\mathcal M_T$).
  \medskip

 Let $G$ be any $\mathcal M_T$-generic filter.
For any $b\in \mathcal M_T$, let $\supp (b)$
denote the countable subset of $T$
called the support of $b$ in the usual sense,
namely that the pull-back of the projection into
 $2^{\supp(b)}$ of $b$ has the same measure
 as $b$.
 \medskip

 The topology $\tau$ on $\kappa\setminus\omega$ is
  generated by declaring $\val_G(\dot W_\rho)$
 to be clopen for every $\rho\in T^+$.
 \medskip

 \begin{lemma} The topology $\tau$ is Hausdorff.
 \end{lemma}
 \begin{proof} Consider any $b\in G$ and
 any distinct $\omega\leq\alpha < \beta<\kappa$.
 Choose   any   $\rho \in T^+$ satisfying
   that $ \rho \restriction\omega $ is not in the set
     $\{ \psi\restriction\omega : \psi\in \supp (b)\}$.
     Of course it follows that $\{\rho\restriction\alpha,
     \rho \restriction\beta\}$ is disjoint from
     the $\supp(b)$.  Now take the condition
     $[ (\rho\restriction\alpha,0)]\cap [(\rho\restriction\beta,1)]$,
      (i.e. the condition $[[\dot W_{\rho}\cap \{\alpha,\beta\}
       = \{\beta\}]]$).  Clearly $b$ meets
       $[[\dot W_{\rho}\cap \{\alpha,\beta\}
       = \{\beta\}]]$ in a set of positive measure,
       and this proves $b$ does not force
       that $\alpha$ and $\beta$ can not be separated
       by a clopen set.
 \end{proof}

 \begin{definition}  For a finite set $F\subset T^+$ and
 a function $\sigma\in 2^F$, we let
  $\dot W_\sigma$ denote the basic clopen set
   $\bigcap \{ \dot W_\rho : \sigma(\rho)=1\}
   \setminus \bigcup\{ \dot W_\rho : \sigma(\rho) =0\}$.\\
   Let $\mbox{Fn}(T,2)$ denote the set of all such finite functions.
 \end{definition}

\begin{lemma}   There is no uncountable left-separated
subspace of $\kappa\setminus \omega$ with respect to the
topology $\tau$.
\end{lemma}

 \begin{proof}  Let, for $\xi\in\omega_1$,
  $\dot \sigma_\xi$ be any name of an element of
   $\mbox{Fn}(T,2)$. It suffices
   to prove that
   $1$ forces that there is some $\zeta\in\omega_1$
   such that $\dot W_{\dot\sigma_\zeta}$ is contained
   in the union of the family
      $\{ \dot W_{\dot\sigma_\xi} : \xi  < \delta\}$;
      which we now do.
      \medskip

First consider any $0<b\in \mathcal M_T$. We prove
there is an extension $b'$ of $b$ and a $\zeta\in\omega_1$
such that $b'$ forces that $\dot W_{\dot \sigma_\zeta}$
is contained
  in the union of the family
      $\{ \dot W_{\dot\sigma_\xi} : \xi  < \delta\}$.

 In the forcing extension  (with $b$ in the generic) we
 can pass to an uncountable subset (and thus re-index)
  so that the family $\{   F_\xi : \xi \in\omega_1\}$,
  where $ F_\xi = \dom (  \sigma_\xi)$,
is a $\Delta$-system and there is a single $  \sigma
\in 2^{<\omega}$
such that each $ F_\xi$ has cardinality equal
to $n= \dom( \sigma)$
 and  the order-preserving function from $n$
 to $F_\xi$ (with the lexicographic ordering)
 sends $\sigma$ to $\sigma_\xi$.

 By extending $b$ we can assume that $b$ decides
 the values of $n,\sigma$ and the root of the
 $\Delta$-system $F$. For every $\xi\in\omega_1$
 (again for the re-indexed subcollection)
 choose an extension $b_\xi $ of $b$ that
 forces $\dot F_\xi = F_\xi$ for some
 $F_\xi \in [T]^n$.
 Notice that, because $\mathcal M_T$ is ccc,
 for each $\xi$, there is a
  $\xi< \eta<\omega_1$ such that $b$
  forces that $F_\xi\setminus F$ is
  disjoint from $\dot F_\zeta$ for all
  $\eta\leq \zeta <\omega_1$.
\medskip

  Therefore, by again passing to an uncountable
  re-indexed subsequence, we can
  assume that $\{ F_\xi : \xi\in\omega_1\}$
  is a $\Delta$-system with root $F$. Additionally,
   since $\mathcal M_T$ is ccc, we can assume
   that, for every $\eta<\omega_1$,
   the join of the family $\{ b_\xi : \eta < \xi\in\omega_1\}$
   is equal to the join of the entire family
    $\{ b_\xi : \xi\in\omega_1\}$.

  Let the family $\{ b_\xi , F_\xi : \xi\in\omega_1\}$
  be an element of a countable elementary submodel
    $M$ of $H((2^\kappa)^+)$.  Let $M\cap\omega_1 = \delta$.
 Now take any $\beta\in\kappa$
    and extension $b'$
     of $b\wedge b_\delta \wedge [[\beta\in \dot W_{\sigma_\delta}]]$.
    We prove, and by genericity this suffices, that
      $b'$ \textbf{does not   force\/}
      that $\beta\notin
       \bigcup \{ \dot W_{\sigma_\xi} : \xi\in \delta\}$.

     \medskip

First we consider the case where   $\beta\in M$.
Choose any generic filter $G$ with $b'\in G$.
 Now   it is a simple application of
the fact that $M[G]$ is an elementary submodel of
  $H(\theta)[G]$ to deduce that there is a $\xi\in M\cap \delta$
   such that $ \beta\in \val_G(\dot W_{\rho_\xi} )$.

\medskip

     Otherwise, let    $\bar\beta$
     be the minimal element of $M\cap [0,\kappa]\setminus \beta$.
It follows that $\bar\beta$ has uncountable cofinality since
$\beta$ is a witness to the fact that $M\cap \bar\beta$
is not cofinal in $\bar\beta$. For each $\xi\leq \delta$,
let $F_\xi = \{ \rho^\xi_i : i < n   \}$ be   listed in lexicographic
order.
Let $L=\{ i : \rho^\delta_i {\restriction} \bar\beta \in M\}$.
Of course $F\subset \{ \rho^\delta_i : i\in L\}$.
\medskip

We observe that $b'\leq b_\delta\wedge [  (\rho^\delta_i\restriction\beta , \sigma(i))]$
for each $i<n$.
Let $L_0 = \sigma^{-1}(0)$ and $L_1 = \sigma^{-1}(1)$.
   Since $b'$ forces $\beta\in \dot W_{\sigma_\delta}$,
   it follows that
    $\{ \rho^\delta_i\restriction \beta : i\in L_0\}$
    is disjoint from
     $\{ \rho^\delta_i \restriction \beta : i\in L_1\}$.
     Since $\beta <\bar\beta$, this implies
     that
$\{ \rho^\delta_i\restriction \bar\beta : i\in L_0\}$
    is disjoint from
     $\{ \rho^\delta_i \restriction \bar\beta : i\in L_1\}$.

Let $ I$ be the set of $\xi\in \omega_1$ such that
\begin{enumerate}
\item $
 \{ \rho^\delta_i  \restriction \bar\beta : i\in L\}
 = \{ \rho^\xi_i \restriction \bar\beta : i\in L\} ~\}$, and
 \item for each $i,j\in n$,
   $\rho^\xi_i\restriction\bar\beta = \rho^\xi_j\restriction\bar\beta$,
   if and only if
   $\rho^\delta_i\restriction\bar\beta = \rho^\delta_j\restriction\bar\beta$.
 \end{enumerate}

 Notice that $I\in M$ and that $\delta\in I$, and
 so $I$ is uncountable. Since $\mathcal M_T$ is ccc,
  there is an $\eta<\omega_1$ such  that the join
  of $\{b_{\xi} : \eta < \xi\in I\}$ is equal to
  the join of $\{ b_{\xi} : \zeta < \xi\in I\}$ for
  all $\eta\leq \zeta\in\omega_1$. By elementarity
  there is such an
   $\eta$  in $M$.
 In addition, by elementarity, for every $\eta<\eta'\in M\cap \delta$,
  the join of the family $\{ b_\xi : \eta'< \xi\in I\cap M\}$
  is equal to the join of the family $\{ b_\xi : \eta<\xi\in I\}$.

 \medskip

Clearly
  $\{ \rho^\delta_i \restriction \bar\beta
  : i\in n\setminus L\}$ is disjoint from $M$.
Therefore,  working in $M $,
we can
 choose $\{\xi_k : k\in \omega\}\subset   I$ and in $M$
 such that, $\{ \rho^{\xi_k}_i\restriction \bar\beta : i\in n\setminus L\}$
 is disjoint from $\{ \rho^{\xi_j} _i \restriction \bar\beta : i < n\}$
 for all $j< k$. We can simultaneously ensure that
 the measure of the join of every co-finite subset
 of $\{ b_{\xi_k} : k\in\omega\}$
 is equal to the join of the family
 $\{ b_\xi : \xi\in I\}$. Similarly, there is a
 partition, $\{ A_\ell : \ell\in \omega\}\subset M$
 of $\omega$ into infinite sets
 such that, for each $\ell$,
 the join of $\{ b_{\xi_k} : k\in A_\ell\}$
 is equal to the join of $\{ b_{\xi_k} : k\in\omega\}$.

 Let, for  $k\in\omega$,   $a_k$ denote the meet
 of the conditions $\{ [ (\rho^{\xi_k}_i {\restriction} \beta,
   \sigma(i) ] : i\in n\setminus L\}$.  There is some $m\leq
   |n\setminus L|$
   such that each $a_k$ has measure $2^{-m}$.
   Also notice
   that the support of each $b_{\xi_k}$ is contained in $M$
   and is therefore disjoint from the support of each $a_{k'}$.

\bgroup
\def\proofname{Proof of Claim}

\begin{claim} For every $c\in M\cap \mathcal M_I$ and finite
 $H\subset \omega$, the measure of  $\bigvee\{ c\wedge a_k : k\in H\}$
 is equal to $\nu(c)\cdot  (1-2^{-m})^{|H|}$.
\end{claim}

\begin{proof}  Since $c$ is in $M$, its support is disjoint from
the support of $a_k$ for all $k\in H$.
The join of $\{ c\wedge a_k : k\in H\}$ is equal to
$c \wedge \bigwedge \{ 1-a_k : k\in H\}$.
The measure of $c \wedge \bigwedge\{ 1-a_k : k\in H\}$
is equal to $\nu(c)\cdot \nu(\bigwedge \{ 1-a_k : k\in H\})$.
The measure of $\bigwedge \{ 1-a_k : k\in H\}$ is equal
to $(1 - 2^{-m})^{|H|}$.
\end{proof}

   \begin{claim}
    The join of the conditions
      $\{ b_{ \xi_{k} }\wedge a_{k  }: k\in \omega\}$
      is equal to the join of
       $\{ b_{\xi_k} : k\in\omega\}$.
   \end{claim}

\begin{proof}
Let $\tilde b = \bigvee \{ b_{\xi_k} : k\in\omega\}$.
   It suffices to prove
  that $\bigvee \{ b_{\xi_{k}}\wedge a_k : k\in \omega\}$
  has measure greater   than $ \nu(\tilde b)\cdot  (1-2^{-m})^{\ell}$
  for each $\ell\in \omega$.
  For each $i\leq \ell$ and each $j\in A_i$,
  $c(i,j) = b_{\xi_{ j }}
  \setminus \bigvee\{ b_{\xi_{ j' }} : j' \in A_i\cap  j\}$.
Thus
  the family $\{ c(i,j) : j\in A_i \}$ is a partition
  of $\tilde b$ into measurable sets.
   Clearly, for each $i\in \omega$
   and $j\in A_i$,
   $c(i,j)\wedge a_j \leq  b_{\xi_j}\wedge a_j$.

   Now let $\mathcal C$ be the family of all
   atoms below $\tilde b$ of the Boolean algebra generated by
   $\bigcup \{ c(i,j) : i\leq \ell, \ j\in A_i\}$.
   In other words, the family $\mathcal C$ is a disjoint refinement
   of $\{ c(i,j) : j\in  A_i  \}$ for every $i\leq \ell$.
   In addition, $\bigvee \mathcal C = \tilde b$.
   For each $c\in \mathcal C$, let $h_c $ be a function
   with domain $\ell+1$ into $\omega$
   satisfying that $h_c(i) = j$ implies that
     $ c \leq c(i,j)$.

     It follows that, for each $c\in \mathcal C$,
      $\bigvee \{ c\wedge a_{h_c(i)}: i \leq\ell\}$
      is less than or equal to
        $\bigvee \{ b_{\xi_k} \wedge a_k : k\in \omega\}$.

The measure of  $\bigvee \{ c\wedge a_{h_c(i)}: i \leq\ell\}
=c \wedge \bigvee_{i\leq\ell} a_{h_c(i)}
$
 is greater than $\nu(c) \cdot (1-2^{-m})^\ell$.
 The measure of $\bigvee \mathcal C$ is $\nu(\tilde b)$
 and so the measure of $\bigvee
 \{ c \wedge \bigvee_{i\leq\ell} a_{h_c(i)} : c\in \mathcal C\}$
 is greater than $\nu(\tilde b)\cdot (1-2^{-m})^\ell$.
 Finally, since  $\bigvee
 \{ c \wedge \bigvee_{i\leq\ell} a_{h_c(i)} : c\in \mathcal C\}$
is less than $\bigvee \{ b_{\xi_k} \wedge a_k : k\in\omega\}$,
this completes the proof of the Claim.
\end{proof}

\egroup

   To complete the proof of the theorem, we just note
   that $  b_{\delta} \leq \bigvee \{ b_{\xi_k} \wedge a_{ k}
   : k\in \omega\}$
and that $r'\leq b_{\delta} \wedge\bigwedge \{ [(\rho^\delta_i, \sigma(i))] :
 i\in L\}$.
   So we may choose $k\in\omega$
   such that $r'\wedge b_{\xi_k}\wedge a_k > 0$.
   This condition
   forces that $ \sigma_{\xi_k} = \dot \sigma_{\xi_k}$ and
    $\beta\in \dot W_{\sigma_{\xi_k}}$.
 \end{proof}

\end{document}